\documentclass[11pt, twoside, leqno]{article}
\usepackage{amsfonts}
\usepackage{amsthm,amssymb,amsmath}
\usepackage{mathrsfs}
\usepackage{t1enc}
\usepackage[cp1250]{inputenc}

\newtheorem{thm}{Theorem}

\theoremstyle{definition}

\newtheorem{exa}[thm]{Example}
\begin{document}

\title{The asymptotic strong Feller property does not imply\\
the e-property for Markov-Feller semigroups}
\author{Joanna Jaroszewska\footnote{Faculty of Mathematics and Natural Science,
Cardinal Stefan Wyszy\'{n}ski University, Warszawa, Poland\newline\indent\hspace{1mm}
\textsl{e-mail address:} \texttt{j.jaroszewska@uksw.edu.pl,} \textsl{{URL:}
\texttt{joannajaroszewska.wordpress.com}}}}
\maketitle

\begin{abstract}
T. Szarek [Stud. Math. 189 (2008), \S\thinspace 4] have discussed the
relationship between two important notions concerning Markov semigroups: the
asymptotic strong Feller property and the e-property, asserting that the
former property implies the latter one. 
%See also a comment on this by T. Szarek, D. Worm [ETDS 32 (2012), §\thinspace 1].
In this very short note we rectify this issue exhibiting a~simple example of
a~Markov-Feller semigroup enjoying the asymptotic strong Feller property, for
which the e-property is not satisfied. (See also the comment on a connection
between the asymptotic strong Feller property and the e-property by T.
Szarek, D. Worm [ETDS 32 (2012), §\thinspace 1]). Additionally we give a very simple
example -- in comparing with the one given by T. Szarek [Stud. Math. 189
(2008), §\thinspace 4] -- showing that also the converse implication does
not hold.
\end{abstract}

%%%%% To ease editing, for IMPAN journals add:

\baselineskip=16.4pt

%%%%%%%%%%%%%%%%

%% Classification and key words; note that the 2010 classification is used:

\renewcommand{\thefootnote}{}

\footnote{%
}


\renewcommand{\thefootnote}{\arabic{footnote}} \setcounter{footnote}{0}

%%%%%%%%

\section{Introduction}

In recent years, a great deal of attention has been focused on properties of
probability measures invariant with respect to Markov semigroups and in
particular on the problem of the uniqueness of such measures. First
important criterion for their uniqueness was given by R. Khas'minski\u{\i}
(see \cite{Kha60} or \cite[Proposition 4.1.1, p.~42]{DaPZab96}), who proved
it for a strongly Feller semigroup under the assumption of its
irreducibility. Throughout the last half of a century, this criterion turned
out to be very useful for applications in analysis of SDEs of various types.
Recently, motivated by the study of properties of the 2D Navier-Stokes
equations with degenerate stochastic forcing, M. Hairer and J. C. Mattingly
have developed the Khas'minski\u{\i}'s ideas introducing
a~brand new Feller-type condition. They have called it an
asymptotic strong Feller condition and used it to show the uniqueness of
invariant probability measures (see \cite{HaiMat06}). The uniqueness of invariant
probability measures has also been examined by T. Szarek in \cite{Sza08},
whose results, in comparing with the Hairer's and Mattingly's work, have taken
advantage of the
e-property and the overlapping supports condition instead of the asymptotic
strong Feller property. His motivation for \cite{Sza08} was connected with
his earlier achievement of proving the existence of an invariant probability
measure for a discrete-time Markov semigroup acting on a Polish space and
possessing the e-property, i.e. being an e-chain (see \cite{Sza06}). Let us mention that in that
way a major breakthrough in the field was made, as earlier there was no
handy and general enough criterion for the existence of an invariant
measure. On the occasion of his research, T. Szarek has discussed the
relation between the asymptotic strong Feller property and the e-property,
claiming in \cite[§ 4]{Sza08} that asymptotic strong Feller operators
possess an equicontinuous dual operator (i.e. the e-property). In a
more recent paper, \cite[§\thinspace 1]{SzaWor12}, its authors have
formulated a weaker assertion : \textquotedblleft it seems that all known
examples of Markov processes with the asymptotic strong Feller property
satisfy the e-property as well\textquotedblright. The aim of this note is
to provide the missing example of a Markov process with the asymptotic
strong Feller property which does not satisfy the e-property. Our example of
such process may be built on every nontrivial normed vector space, no matter
finite or infinite dimensional. The example is based on the construction of
a discontinuous nonlinear solution to the Cauchy equation given by G. Hamel
in \cite{Ham05} and on pieces of information about this solution which may
be found in \cite{RieSah98}.

The rest of the paper is divided into two sections: in first of them we
introduce our notation and terminology and in the second one, in Theorem \ref%
{Th:Main}, we present the example of a Markov-Feller semigroup with the
asymptotic strong Feller property, for which the e-property does not hold.
Actually we show that this example does not satisfy the asymptotic
e-property, which is a variant of the e-property, slightly general than the
original condition. Let us mention that the asymptotic e-property has been
introduced in \cite{Jar13a} and it seems that it may be a convenient tool to
examine Feller as well as some non-Feller processes (for a recent study of
non-Feller processes see \cite{Jar13b}). In the last section we also put a
converse example of a~Markov-Feller semigroup with the
e-property and without the asymptotic strong Feller property.

\section{Preliminaries}

As we want to make this note self-contained, we recall all the notions
concerning properties of Markov semigroups which we will need in the sequel.
The basic structure for our study is a nonempty metric space $(X,\varrho )$
and the set $\mathcal{M}_{\varrho }(X)$ of all finite positive Borel
measures on $X$. These measures will be used to integrate functions
from the set $B_{b}(X)$ of all real bounded Borel functions and from its
subset $C_{b}(X)$ of all real bounded continuous functions. A~family $\left( %
\mathscr{P}_{t}\right) _{t\geq 0}$ is called a \textit{Markov semigroup}
provided it is a~semigroup of positively homogeneous and additive maps $%
\mathscr{P}_{t}:\mathcal{M}_{\varrho }(X)\rightarrow \mathcal{M}_{\varrho
}(X)$ which preserve measures, i.e. the equality $\mathscr{P}_{t}\mu (X)=\mu
(X)$ is fulfilled for every $\mu \in \mathcal{M}_{\varrho }(X)$ and $t\geq 0$%
. A~Markov semigroup $\left( \mathscr{P}_{t}\right) _{t\geq 0}$\ is called a 
\textit{Markov-Feller} semigroup provided there exists a semigroup $\left( %
\mathscr{U}_{t}\right) _{t\geq 0}$ of maps $\mathscr{U}_{t}:B_{b}(X)%
\rightarrow B_{b}(X)$ such that $\mathscr{U}_{t}C_{b}(X)\subset C_{b}(X)$\
and $\left( \mathscr{U}_{t}\right) $ dual to $\left( \mathscr{P}_{t}\right) $%
. This last condition means that $\int_{X}\mathscr{U}_{t}\varphi \,\mathrm{d}%
\mu =\int_{X}\varphi \,\mathrm{d}\mathscr{P}_{t}\mu $ for every $t\geq 0$, $%
\mu \in \mathcal{M}_{\varrho }(X)$ and $\varphi \in B_{b}(X)$. A~useful
example of a Markov semigroup is a~\textit{Perron-Frobenius semigroup}
$\left( \mathscr{P}_{t}\right) _{t\geq 0}$
corresponding to a given semigroup of Borel maps $\left( S_{t}:X\rightarrow
X\right) _{t\geq 0}$, i.e. such a family that
\begin{equation}
\mathscr{P}_{t}\mu (A)=\mu \left( {S_t}^{-1} (A) \right)
\hspace{0.3in}\text{for }t\geq 0,\mu \in \mathcal{M}_{\varrho }(X),A\in 
\mathcal{B}_{X},  \label{Eq:PF}
\end{equation}%
where $\delta _{y}$ denotes, as usual, the probabilistic measure
concentrated at $y\in X$ and $\mathcal{B}_{X}$ -- the family of all Borel
subsets of $X$. We will use \eqref{Eq:PF} later on. Notice that if
a~semigroup $\left( S_{t}:X\rightarrow X\right) _{t\geq 0}$ consists of
continuous maps, then the corresponding Perron-Frobenius semigroup is
a~Markov-Feller one (see Chapter 12 of \cite{LasMac94}).

Now we recall the stability-related concepts which we would like to compare.
We start with two conditions concerning equicontinuity. In what follows $%
L_{b}(X)$ denotes the set of all real Lipschitz-continuous bounded
functions. A Markov semigroup $\left( \mathscr{P}_{t}:\mathcal{M}_{\varrho
}(X)\rightarrow \mathcal{M}_{\varrho }(X)\right) _{t\geq 0}$ is said to have
the \textit{e-property }at $y\in X$ if for every $f\in L_{b}(X)$, the family 
$\left( \mathscr{U}_{t}f\right) _{t\geq 0}$ is equicontinuous at $y\in X$,
i.e. 
\begin{equation*}
\lim_{x\rightarrow y}\sup_{t\geq 0}\left\vert \mathscr{U}_{t}f(y)-\mathscr{U}%
_{t}f(x)\right\vert =0.
\end{equation*}%
It is said to have the \textit{asymptotic e-property }at $y\in X$ if for
every $f\in L_{b}(X)$ 
\begin{equation}
\lim_{x\rightarrow y}\limsup_{t\rightarrow \infty }\left\vert \mathscr{U}%
_{t}f(y)-\mathscr{U}_{t}f(x)\right\vert =0.  \label{Eq:AEq}
\end{equation}%
We say that the semigroup $\left( \mathscr{P}_{t}\right) _{t\geq 0}$ has the 
\textit{(asymptotic) e-property }provided the appropriate property holds at
every $y\in X$. If a Markov semigroup possess the e-property, the Markov chain
it generates is called an \textit{e-chain}. This last name was popularized 
thanks to \cite{MeyTwe09}.
Clearly, if a Markov semigroup $\left( \mathscr{P}_{t}\right) _{t\geq 0}$ has
the e-property at $y\in X$, it has the asymptotic e-property at $y$ as well.
The converse implication does not hold, see \cite[Example 7]{Jar13a}.

Now we give, after \cite{HaiMat06}, the definition of the asymptotic strong
Feller property. Firstly we need to introduce some technical notions
concerning pseudo-metrics. For a pseudo-metric $\sigma $ on $X$, by $%
B_{\sigma }(x,\gamma )$ we denote the open $\sigma $-ball with the radius $%
\gamma >0$ and the center $x\in X$, i.e. $B_{\sigma }(x,\gamma )=\{y\in
X:\sigma (x,y)<\gamma \}$. By $\sigma ^{y}$ we denote the function $\sigma
^{y}:X\rightarrow \mathbb{R}$ defined by the formula 
\begin{equation*}
\sigma ^{y}(x)=\sigma (x,y)\hspace{0.3in}\text{for }x\in X. 
\end{equation*}%
Notice that for every $y\!\in\! X$, $\sigma^{y}\!\in\! Lip_{\sigma }(1)$, where 
$$Lip_{\sigma }(1)=\left\{ \varphi\! :\!X\!\rightarrow\!\mathbb{R}|
\forall\,x,y\!\in\! X:\,\left\vert \varphi (x)\!-\!\varphi (y)\right\vert\leq
\sigma (x,y) \right\}.$$ 
%Notice that for every $y\in X$, $\sigma ^{y}\in Lip_{\sigma }(1)$, where $Lip_{\sigma }(1)=\left\{ \varphi :X\rightarrow \mathbb{R}| \forall\,x,y\in X:\,\left\vert \varphi (x)-\varphi (y)\right\vert\leq \sigma (x,y) \right\} $. 
An increasing sequence $\left( \varrho _{n}:X\times
X\rightarrow \lbrack 0,\infty )\right) _{n\in \mathbb{N}}$ of pseudo-metrics
continuous with respect to $(X,\varrho )$ is called a \textit{totally
separating system of pseudo-metrics} if $\varrho _{n}(x,y)\underset{%
n\rightarrow \infty }{\longrightarrow }1$ for every $x\neq y$. Every $%
\varrho $-continuous pseudo-metric $\sigma $ on $X$ may be extended to $%
\mathcal{M}_{\varrho }(X)$ by the formula%
\begin{equation*}
\left\Vert \mu _{1}-\mu _{2}\right\Vert _{\sigma }=\sup_{\varphi \in
Lip_{\sigma }(1)}\left\vert \int_{X}\varphi \,\mathrm{d}(\mu _{1}-\mu
_{2})\right\vert .
\end{equation*}%
This pseudo-metric on $\mathcal{M}_{\varrho }(X)$ is well known under the
name the \textit{Wasserstein pseudo-metric}.

A Markov semigroup $\left( \mathscr{P}_{t}:\mathcal{M}_{\varrho
}(X)\rightarrow \mathcal{M}_{\varrho }(X)\right) _{t\geq 0}$ of operators
defined for a~metric space $\left( X,\varrho \right) $ is said to have the 
\textit{asymptotic strong Feller property} at $y\in X$ provided there are: a
sequence $\left( t_{n}\right) _{n\in \mathbb{N}}$ of positive reals and a
totally separating system of pseudo-metrics $\left( \varrho _{n}\right)
_{n\in \mathbb{N}}$ on $X$ such that%
\begin{equation*}
\lim_{\gamma \rightarrow 0^{+}}\limsup_{n\rightarrow \infty }\sup_{x\in
B_{\varrho }(y,\gamma )}\left\Vert \mathscr{P}_{t_{n}}\delta _{y}-\mathscr{P}%
_{t_{n}}\delta _{x}\right\Vert _{\varrho _{n}}=0.
\end{equation*}%
It is said to have the \textit{asymptotic strong Feller property} if the
above condition holds at every $y\in X$.

In practice (see e.g. \cite{HaiMat06,HenMatWoo11,Xu11}), to verify that
certain Markov semigroup satisfies the asymptotic strong Feller property one
usually takes $t_{n}=n$ and $\varrho _{n}=1\wedge a_{n}\varrho $ for some
fixed sequence $\left( a_{n}\right) $ of positive reals with $a_{n}\uparrow
\infty $. Our choice of $\left( t_{n}\right) $ and $\left( \varrho
_{n}\right) $ in the proof of Theorem \ref{Th:Main} is similar to this one.

We end this section with providing some information on Hamel bases which
will be needed in our proof of Theorem \ref{Th:Main}. A set $\mathbb{B}%
\subset \mathbb{R}$ is called \textit{a Hamel basis} for $\mathbb{R}$ if
every element of $\mathbb{R}$ is a unique finite rational linear combination
of elements of $\mathbb{B}$. The existence of a Hamel basis is guaranteed by
the axiom of choice (see e.g. \cite{Ham05}, \cite[p. 132]{Tay85}). The Hamel
bases are useful for constructions of functions with nontypical properties,
as the following observation, taken from \cite[Theorem 1.6, p.10]{RieSah98},
shows.

\begin{thm}
\label{Th:Ham}If $\mathbb{B}$ is a Hamel basis for $\mathbb{R}$ and $g:%
\mathbb{B}\rightarrow \mathbb{R}$ is an arbitrary function then there exists
a function $f_{g}:\mathbb{R}\rightarrow \mathbb{R}$ which satisfies the
Cauchy equation (i.e. $f_{g}(x)+f_{g}(y)=f_{g}(x+y)$ for all $x,y\in \mathbb{%
R}$, i.e. $f_{g}$ is additive) and such that $f_{g}|_{\mathbb{B}}=g|_{%
\mathbb{B}}$.
\end{thm}

We will use this theorem below. Notice that the extension $f_{g}:\mathbb{R}%
\rightarrow \mathbb{R}$ of a~given function $g:\mathbb{B}\rightarrow \mathbb{%
R}$ is unique, which is guaranteed by \cite[Theorem 1.5, p.9]{RieSah98}.

\section{Examples of Markov-Feller semigroups}

\begin{thm}
\label{Th:Main}If $\left( X,\left\Vert \cdot \right\Vert \right) $ is a
normed vector space over $\mathbb{R}$ or $\mathbb{C}$, $X\neq \{0\}$ and $%
\varrho $ is a~metric induced by $\left\Vert \cdot \right\Vert $ then there
exists a Markov-Feller semigroup $\left( \mathscr{P}_{t}\right. :\mathcal{M}_{\varrho }(X)\rightarrow \left. 
\mathcal{M}_{\varrho }(X)\right) _{t\geq 0}$ which
does have the asymptotic strong Feller property and does not have the
asymptotic e-property (and a fortiori the e-property) at every $y\in X$.
\end{thm}

\begin{proof}
Let $\mathbb{B}\subset \mathbb{R}$ be a Hamel basis for $\mathbb{R}$. Fix
any $b_{1},b_{2}\in \mathbb{B}$ and $g:\mathbb{B}\rightarrow \mathbb{R}$
such that $b_{1}\cdot g(b_{1})<0<b_{2}\cdot g(b_{2})$. Let $f_{g}$ be the
additive extension of $g$, existing by Theorem \ref{Th:Ham}. Notice that
since $\frac{f_{g}(|b_{1}|)}{|b_{1}|}<0<\frac{f_{g}(|b_{2}|)}{|b_{2}|}$, $%
f_{g}$ is nonlinear and hence (highly) discontinuous.

Let $\left( X,\left\Vert \cdot \right\Vert \right) $ be any fixed nontrivial
(i.e. $\neq \{0\}$) normed vector space over $\mathbb{R}$ or $\mathbb{C}$
and let $\varrho $ be a metric induced by $\left\Vert \cdot \right\Vert $.
Next, let%
\begin{equation*}
S_{t}(x)=e^{f_{g}(|b_{1}|t)}x\hspace{0.3in}\text{for }t\geq 0,x\in X.
\end{equation*}%
As $f_{g}$ is additive, $\left( S_{t}:X\rightarrow X\right) _{t\geq 0}$ is a
semigroup and hence it generates, via \eqref{Eq:PF}, the Perron-Frobenius
semigroup $\left( \mathscr{P}_{t}\right) _{t\geq 0}$. Since $S_{t}$ is
continuous for each $t\geq 0$, $\left( \mathscr{P}_{t}\right) _{t\geq 0}$ is
a Markov-Feller semigroup, with the dual semigroup $\left( \mathscr{U}%
_{t}\right) _{t\geq 0}$ given by the formula%
\begin{equation*}
\mathscr{U}_{t}\varphi (x)=\varphi \left( e^{f_{g}(|b_{1}|t)}x\right) 
\hspace{0.3in}\text{for }t\geq 0,x\in X,\varphi \in B_{b}(X).
\end{equation*}

Firstly we shall show that $\left( \mathscr{P}_{t}\right) _{t\geq 0}$
satisfies the asymptotic strong Feller property with $t_{n}=n$ and $\varrho
_{n}=1\wedge n\varrho $. As $Lip_{\varrho _{n}}(1)\subset \{\varphi |\frac{%
\varphi }{n}\in Lip_{\varrho }(1)\}$, we obtain 
\begin{equation*}
\left\Vert \mathscr{P}_{n}\delta _{y}-\mathscr{P}_{n}\delta _{x}\right\Vert
_{\varrho _{n}}\leq \sup_{\psi \in Lip_{\varrho }(1)}n\left\vert \mathscr{U}%
_{n}\psi (x)-\mathscr{U}_{n}\psi (y)\right\vert \leq
ne^{nf_{g}(|b_{1}|)}\varrho (x,y)
\end{equation*}%
for every $x,y\in X$. Hence, since $f_{g}(|b_{1}|)<0$, 
\begin{equation*}
\sup_{x\in B_{\varrho }(y,\gamma )}\left\Vert \mathscr{P}_{n}\delta _{y}-%
\mathscr{P}_{n}\delta _{x}\right\Vert _{\varrho _{n}}\leq \gamma
ne^{nf_{g}(|b_{1}|)}\underset{n\rightarrow \infty }{\longrightarrow }0
\end{equation*}%
for every $y\in X$ and the asymptotic strong Feller property at arbitrarily
chosen $y\in X$ follows.

Now we shall examine the asymptotic e-property. Let arbitrary $y\in X$ be
fixed together with a sequence $\left( x_{n}\right) $ of points converging
to $y$ such that $\left\Vert y\right\Vert \neq \left\Vert x_{n}\right\Vert $
for every $n\in \mathbb{N}$. We check that \eqref{Eq:AEq} does not hold for $%
f=\varrho ^{0}$. As every neighbourhood of $y$ contains some $x_{n}$ and for
every $n\in \mathbb{N}$ we have%
\begin{eqnarray*}
\limsup_{t\rightarrow \infty }\left\vert \mathscr{U}_{t}\varrho ^{0}(y)-%
\mathscr{U}_{t}\varrho ^{0}(x_{n})\right\vert  &\geq &\limsup_{t\uparrow
\infty ,t\in \left\{ m\left\vert \frac{b_{2}}{b_{1}}\right\vert :m\in 
\mathbb{N}\right\} }\left\vert \mathscr{U}_{t}\varrho ^{0}(y)-\mathscr{U}%
_{t}\varrho ^{0}(x_{n})\right\vert  \\
&=&\left\vert \left\Vert y\right\Vert -\left\Vert x_{n}\right\Vert
\right\vert \limsup_{m\rightarrow \infty }e^{mf_{g}(|b_{2}|)}=\infty ,
\end{eqnarray*}%
the asymptotic e-property does not hold at every $y\in X$.
\end{proof}

\begin{exa}
We end this note with the very simple example of a Markov-Feller semigroup,
which, conversely to the semigroup from the above proof of Theorem \ref%
{Th:Main}, does have the e-property and does not satisfy the asymptotic
strong Feller property. It is a simplest possible example: the semigroup of
identity operators -- please compare it with the example from \cite[§%
\thinspace 4]{Sza08}. Fix a metric space $(X,\varrho )$ with $l(X)\neq
\emptyset $, where $l(X)$ denotes the set of all limit points of $X$, i.e.
the set consisted of all such $x\in X$ that every neighbourhood of $x$
contains at least two distinct points: $x$ and another one. Consider $\left( %
\mathscr{P}_{t}\right) _{t\geq 0}$ defined by the formula 
\begin{equation}
\mathscr{P}_{t}\mu =\mu \hspace{0.3in}\text{for }t\geq 0,\mu \in \mathcal{M}%
_{\varrho }(X).  \label{Eq:triv}
\end{equation}%
As identity maps $\mathscr{U}_{t}=id_{B_{b}(X)}$ ($t\geq 0$) form a
semigroup dual to $\left( \mathscr{P}_{t}\right) _{t\geq 0}$, this last
semigroup is a Markov-Feller one and satisfies the e-property. To check that
the asymptotic strong Feller property does not hold at every point belonging
to $l(X)$, we fix any sequence $\left( t_{n}\right) $ of positive reals and
a totally separating system of pseudo-metric $\left( \varrho _{n}\right) $.
Next choose arbitrarily $y\in l(X)$ and $\gamma >0$. We shall see that 
\begin{equation}
\limsup_{n\rightarrow \infty }\sup_{x\in B_{\varrho }(y,\gamma )}\left\Vert %
\mathscr{P}_{t_{n}}\delta _{y}-\mathscr{P}_{t_{n}}\delta _{x}\right\Vert
_{\varrho _{n}}\geq \frac{1}{2}.  \label{Eq:id}
\end{equation}%
Indeed, since $\varrho _{n}^{y}\in Lip_{\varrho _{n}}(1)$ and there exist $%
z\in B_{\varrho }(y,\gamma )$ and $n_{0}\in \mathbb{N}$ such that $\varrho
_{n}(z,y)\geq \frac{1}{2}$ for $n\geq n_{0}$, we have%
\begin{equation*}
\sup_{x\in B_{\varrho }(y,\gamma )}\left\Vert \mathscr{P}_{t_{n}}\delta _{y}-%
\mathscr{P}_{t_{n}}\delta _{x}\right\Vert _{\varrho _{n}}\geq \sup_{x\in
B_{\varrho }(y,\gamma )}\left\vert \,\varrho _{n}^{y}(x)-\varrho
_{n}^{y}(y)\right\vert \geq \frac{1}{2}\hspace{0.3in}\text{for }n\geq n_{0}.
\end{equation*}%
Hence \eqref{Eq:id} follows. The lack of the asymptotic strong Feller property
is not surprising here provided we take into account the smoothing character of this condition, 
together with the motivation for this notion and its relation to the strong Feller property (see \cite{Hai08}). Clearly our
semigroup \eqref{Eq:triv} is not strong Feller as well (this is also not strange
-- see \cite[Remark 3.10]{HaiMat06}).
\end{exa}

\vskip 20mm

%\bibliographystyle{elsart-num2}
%\bibliographystyle{amsalpha}
%\bibliographystyle{amsry}
%\bibliographystyle{amsplain}
%\bibliography{ajjbib,ajbib,aj1bib}

\providecommand{\bysame}{\leavevmode\hbox to3em{\hrulefill}\thinspace}
%\providecommand{\MR}{\relax\ifhmode\unskip\space\fi MR }
%\providecommand{\MRhref}[2]{%
%  \href{http://www.ams.org/mathscinet-getitem?mr=#1}{#2} }
%\providecommand{\href}[2]{#2}

\end{document}